\documentclass[11pt,a4paper]{article}
\title{\bf Wasserstein Convergence Rates for Empirical Measures of Subordinated Processes on Noncompact Manifolds}
\author{Huaiqian Li\footnote{Email: {\color{blue}huaiqianlee@gmail.com}}
\quad Bingyao Wu\footnote{Email: {\color{blue}bingyaowu@163.com}}
  \vspace{2mm}
\\
{\footnotesize Center for Applied Mathematics, Tianjin University, Tianjin 300072, P. R. China}
}

\date{}
\usepackage{amssymb,amsmath,amsfonts,amsthm,color,mathrsfs}
\usepackage{bbm}
\usepackage[marginal]{footmisc}
\usepackage[]{hyperref}
\usepackage{color} 
\def\R{\mathbb{R}}
\def\E{\mathbb{E}}
\def\P{\mathbb{P}}
\def\Pp{\mathscr{P}}

\def\d{\textup{d}}

\def\<{\langle}
\def\>{\rangle}
\def\Proof.{\noindent{\bf Proof. }}

\newcommand{\vv}{\varepsilon}
\newcommand{\pp}{\partial}
\def\ll{\lambda}
\def\aa{\alpha}
\def\gg{\gamma}
\def\dd{\delta}
\def\nn{\nabla}
\def\C{\mathscr{C}}
\def\W{\mathbb{W}}
\def\e{\text{\rm{e}}}

\def\N{\mathbb{N}}

\def\ff{\frac}

\def\newdot{{\kern.8pt\cdot\kern.8pt}}

\newtheorem{theorem}{Theorem}[section]

\newtheorem{corollary}[theorem]{Corollary}

\newtheorem{example}[theorem]{Example}

\theoremstyle{definition}\newtheorem{remark}[theorem]{Remark}

\begin{document}
\allowdisplaybreaks
\maketitle
\makeatletter 
\renewcommand\theequation{\thesection.\arabic{equation}}
\@addtoreset{equation}{section}
\makeatother 

\begin{abstract}
The asymptotic behaviour of empirical measures has been studied extensively. In this paper, we consider empirical measures of given subordinated processes on complete (not necessarily compact) and connected Riemannian manifolds with possibly nonempty boundary. We obtain rates of convergence for empirical measures to the invariant measure of the subordinated process under the Wasserstein distance. The results, established for more general subordinated processes than [arXiv:2107.11568], generalize the recent ones in [Stoch. Proc. Appl. 144(2022), 271--287] and are shown to be sharp by a typical example. The proof is motivated by the aforementioned works.
\end{abstract}



\section{Introduction}\hskip\parindent
Consider a $d$-dimensional complete and  connected Riemannian manifold $M$ with possibly nonempty boundary $\pp M$.
Let $V\in C^1(M)$ and $Z_V:=\int_M \e^{V(x)}\,\d x$ such that $Z_V<\infty$, where $\d x$ denotes the Riemannian volume measure.
Set $\mu(\d x):=Z_V^{-1}\e^{V(x)}\d x$, which is clearly a Borel probability measure on $M$. Let $L^p(\mu)$ be the usual $L^p$ space over $(M,\mu)$ with norm $\|\cdot\|_{L^p(\mu)}$ for every $p\in[1,\infty]$, and let $\Pp$ be the set of all Borel probability measures on $M$.

Let $(X_t)_{t\ge 0}$ be the diffusion process on $M$ corresponding to the infinitesimal generator $L:=\Delta+\nabla V$ with domain $\mathcal{D}(L)$ in $L^2(\mu)$, where $\Delta$ is the Laplace--Beltrami operator on $M$ and $\nabla$ is the Riemannian gradient. If $\pp M\neq \emptyset$, then we assume that $(X_t)_{t\geq0}$ is reflected at $\partial M$ or satisfies the Neumann boundary condition. Let $(P_t)_{t\ge 0}$ be the (Neumann) Markov semigroup or heat flow corresponding to $(L,\mathcal{D}(L))$, where for every bounded measurable function $f$ on $M$,
$$P_tf(x):=\E^x f(X_t),\quad t\ge 0,\, x\in M.$$
Here and in the sequel, $\E^x$ denotes the expectation for the corresponding process with initial point $x\in M$.  It is well known that the diffusion process $(X_t)_{t\geq0}$ is reversible with the stationary distribution $\mu$. In other words, $(P_t)_{t\ge 0}$ is symmetric in $L^2(\mu)$.

In order to introduce the subordinated process, recall that $B$ is a Bernstein function if
$$B\in C^\infty((0,\infty);[0,\infty))\cap C([0,\infty);[0,\infty)),$$
and for each $n\in \N$,
$$(-1)^{n-1}\ff {\d ^n}{\d r^n}B(r)\ge 0,\quad r>0.$$
The following particular class of Bernstein functions is much more interesting, i.e., 
\begin{equation*}\begin{split}
&\mathbf{B}:=\{B: B\mbox{ is a Bernstein function with }B(0)=0,\,B'(0)>0\}.
\end{split}\end{equation*}
For each $B\in\mathbf{B}$, it is well known that  there exists a unique subordinator $(S_t^B)_{t\geq0}$ associated with $B$, i.e., a one-dimensional, increasing process with independent, stationary increments such that $t\mapsto S_t^B$ is continuous in probability and $S_0^B=0$, characterized by the Laplace transform
\begin{equation}\label{BL}
\E\e^{-\ll S_t^B}=\e^{-tB(\ll)},\quad t,\ll\ge 0.
\end{equation}
Moreover, we will need the following two classes of Bernstein functions, i.e.,
$$
\mathbf{B}^\aa:=\Big\{B\in\mathbf{B}: \liminf_{\ll\to\infty}\ff{B(\ll)}{\ll^{\aa}}>0\Big\},\quad
\mathbf{B}_\aa:=\Big\{B\in\mathbf{B}: \limsup_{\ll\to\infty}\ff{B(\ll)}{\ll^{\aa}}<\infty\Big\},
$$
where $\aa\in[0,1]$. Note that, unlike \cite{WangWu}, any $B$ from $\mathbf{B}^\aa$ or $\mathbf{B}_\aa$ is not required to satisfy
\begin{equation}\label{B-int}
\int_1^\infty r^{d/2-1}\e^{-t B(r)}\,\d r<\infty,\quad t>0.
\end{equation}
However, in some occasions, we will assume the ultra-contractivity of $(P_t^B)_{t>0}$, which implies \eqref{B-int} in certain situations; see
\eqref{UC} and Remark \ref{rk-UC}(ii) below.
\begin{remark}
Let $\aa\in[0,1]$. Recently, in \cite{WangWu}, the following classes of Bernstein functions are defined, i.e.,
$$\mathbb{B}:=\{B\in\mathbf{B}: B\mbox{ satisfies }\eqref{B-int}\},$$
and
$$\mathbb{B}^\aa:=\Big\{B\in\mathbb{B}: \liminf_{\ll\to\infty}\ff{B(\ll)}{\ll^{\aa}}>0\Big\},\quad
\mathbb{B}_\aa:=\Big\{B\in\mathbb{B}: \limsup_{\ll\to\infty}\ff{B(\ll)}{\ll^{\aa}}>0\Big\}.$$
Note that by the elementary inequality \eqref{LB} below, it is immediate to prove that  $\mathbf{B}^\aa=\mathbb{B}^\alpha$, $\aa\in(0,1]$. However, for any $\lambda\geq0$, letting
$$B_1(\lambda):=1-(1+\lambda)^{\alpha-1}, \quad\alpha\in[0,1),$$
and
$$B_2(\lambda):=\frac{\lambda}{1+\lambda},$$
 we can easily check that
  \begin{itemize}
  \item[(1)] $B_1\in\mathbf{B}_\alpha$ for all $\alpha\in[0,1)$, $B_1\in\mathbf{B}^0$, and $B_1\notin\mathbb{B}$;

  \item[(2)] $B_2\in\mathbf{B}_\alpha$ for all $\alpha\in[0,1]$, $B_2\notin\mathbb{B}$, $B_2\in \mathbf{B}^0$,  and  $B_2\notin \mathbf{B}^\alpha$ for every $\alpha\in(0,1]$.
\end{itemize}
Due to this, we conclude that $\mathbb{B}_\alpha\varsubsetneq\mathbf{B}_\alpha$ for every $\alpha\in[0,1]$. For other examples, refer to \cite[Chapter 16]{SSV2012}.
\end{remark}

For every $B\in\mathbf{B}$, let $(X_t^B)_{t\geq0}$ with $X_t^B:=X_{S_t^B}$ be the Markov process on $M$, where $(S_t^B)_{t\ge 0}$ is a stable process as above,  independent of $(X_t)_{t\ge 0}$, satisfying \eqref{BL}.  We call $(X_t^B)_{t\geq0}$ the $B$-subordinated process to $(X_t)_{t\geq0}$.
Let $(P_t^B)_{t\geq0}$ be the Markov semigroup or heat flow corresponding to $(X_t^B)_{t\ge 0}$. It is well known that the infinitesimal generator of $(X_t^B)_{t\geq0}$ is $B(-L)$; see e.g. \cite[Chapter 13.2]{SSV2012} and \cite{Phi}. Note that, in particular, if $M=\R^d$, $\aa\in(0,1)$ and  $B(t)=t^\aa$,  then $(X_t^B)_{t\geq0}$ is the well known $2\aa$-stable process. See e.g. \cite{Bertoin97,SSV2012} for a comprehensive study on Bernstein functions and subordinated processes.

For every $B\in\mathbf{B}$, let us define the empirical measures associated with the $B$-subordinated process $(X_t^B)_{t\geq0}$, i.e.,
$$\mu_t^B:=\ff 1 t \int_0^t\delta_{X_s^B}\,\d s,\quad t>0,$$
where $\delta_\cdot$ is the Dirac measure.

We use $\rho$ to denote the Riemannian distance  on $M$. For every $p\in(0,\infty)$, the $L^p$-Wasserstein (or Kantorovich) distance $\W_p$ is the pseudo-distance between two probability measures on $M$ induced by $\rho$, i.e.,
$$\W_p(\mu_1,\mu_2):=\inf_{\pi\in\C(\mu_1,\mu_2)}\left(\int_{M\times M}\rho(x,y)^p\, \pi(\d x,\d y)\right)^{\ff 1{p\vee 1}},\quad \mu_1,\mu_2\in\Pp,$$
where $\C(\mu_1,\mu_2)$ is the set of all Borel probability measures on the product space $M\times M$ whose marginal distributions are $\mu_1$ and $\mu_2$, respectively. Each probability measure $\pi$ from $\C(\mu_1,\mu_2)$  is also called a coupling measure or coupling of $\mu_1$ and $\mu_2$. More accurately, $\W_p$ with $p\in(0,1)$ should be called Zolotarev distance (see \cite{Zol}).  Refer to \cite[Chapter 7]{Vi2003} and \cite[Chapter 5]{ChenMF2004} for further details on the $L^p$-Wasserstein distance.

On the one hand, the study of large time behaviours of empirical measures is important. It is well known that if a Markov process on some Polish space is stationary and ergodic, then by the strong law of large numbers, almost surely, the empirical measures associated with the process goes to the unique invariant measure weakly as $t\to\infty$.  It is an interesting and always challenging problem to quantify this kind of large time behaviours. On the other hand, the archetypal $\alpha$-stable process ($\alpha\in(0,2)$), a particular L\'{e}vy process or L\'{e}vy flight, has been investigated intensively in various areas. From an applied perspective, the $\alpha$-stable process or the fractional Laplacian has been widely employed to model the real-world phenomena, in particular those exhibiting discontinuous paths or having heavy-tailed distributions, from physics to finance, see e.g. \cite{KA,SW} and references therein.

In the present work, we mainly aim to  study  the rate of convergence of $\mu_t^B$ to $\mu$ under the $L^2$-Wasserstein distance on average, i.e., $\E[\W_2(\mu_t^B,\mu)]$,  for large enough $t$. On this topic, a series of works has appeared recently. We briefly mention them here. In the particular $B(r)=r$ case,  when $M$ is compact, see \cite{WangZhu} and see \cite{LW1,LW2,Wang20201,Wang20211} for further investigations on the case of conditional empirical measures associated with (subordinated) diffusion processes absorbed at the boundary $\partial M$, while when $M$ is not necessarily compact, refer to \cite{Wang20202}. In the case of more regular $B$-subordinated processes mentioned above, when $M$ is assumed to be compact, see the very recent work \cite{WangWu}. See also \cite{Wang20212} in the setting of semilinear stochastic partial differential equations. We should mention that, when $M$ is compact, besides rates of convergence, precise limits are obtained in the main results of the aforementioned papers \cite{WangZhu,Wang20201,Wang20211,WangWu,LW1,LW2}.

In the sequel,  we use $\E^\nu$ and $\P^\nu$ to denote the  expectation and probability measure for the corresponding process with initial distribution $\nu\in\Pp$, respectively. For every $\nu\in \Pp$ and every $t\ge 0$, let $\nu P_t:=\P^\nu(X_t\in\cdot)$ be the distribution of $X_t$ with initial distribution $\nu$. Let $\P_{S_t^B}$ denote the distribution of $S_t^B$. For every $r>0$ and every $x\in M$, $B(x,r)$ denotes the open ball in $M$ with radius $r$ and center $x$.

The rest of this paper is organized as follows. In Section 2, we introduce our main results. In Section 3, an example is given to illustrate the main results. Proofs of the main results and the example are presented in Sections 4, 5 and 6, respectively. An appendix is also included.

\section{Preparations and Main results}\hskip\parindent
In this section, we present the main results. For this purpose, we should introduce further notions and notations which will be frequently used below.

Let $(p_t)_{t>0}$ be the heat kernel of $(P_t)_{t>0}$ w.r.t. $\mu$. Set 
\begin{equation}\label{DHU}
\gg(t):=\int_M p_t(x,x)\,\mu(\d x)<\infty,\quad t>0,
\end{equation}
which will be in force throughout  the paper. It is pointed out in \cite[page 272]{Wang20202} that,  as the result shown in \cite[Theorem 3.3]{Wang2000} or \cite[Theorem 3.3.19]{Wang2005}, by the representation formula on the heat kernel $(p_t)_{t>0}$ (see \eqref{HK} below),
\eqref{DHU} is equivalent to that $L$ has empty essential spectrum  such that,  all the eigenvalues of $-L$ counting multiplicities, listed in
increasing order, denoted by  $\{\ll_i\}_{i\ge 0}$,  are nonnegative and satisfy that $\lambda_i\rightarrow\infty$ as $i$ tends to $\infty$ and
$$\sum_{i=0}^\infty\e^{-\ll_i t}<\infty,\quad t>0.$$
\begin{remark}\label{remark-1}
In order to guarantee that $L$ has only discrete spectrum, it is sufficient to assume that $\gamma(t_0)<\infty$ for some $t_0>0$. For the proof, see Appendix.
\end{remark}

Assume \eqref{DHU}. Since $M$ is connected, $\lambda_0$, which equals zero, is the algebraically simple and isolated eigenvalue of $-L$. 
Consequently, $L$ has a spectral gap, i.e.,
\begin{equation}\label{SG}
\ll_1:=\inf\big\{\mu(|\nn f|^2):f\in\mathcal{D}(L),\,\mu(f)=0,\,\mu(f^2)=1\big\}>0.
\end{equation}
It is well known that \eqref{SG} is equivalent to the Poincar\'{e} inequality (see e.g. \cite[Theorem 1.1]{Wang2005}), i.e.,
$$\|P_tf-\mu(f)\|_{L^2(\mu)}\leq \e^{-\lambda_1 t}\|f-\mu(f)\|_{L^2(\mu)},\quad t\geq0,\,f\in L^2(\mu).$$
Here and in what follows, we write $\mu(f)$ as the shorthand notation for $\int_M f\,\d\mu$.

For every $\vv>0$, let
\begin{equation*}\label{DE}
\dd(\vv):=\E^\mu[\rho(X_0,X_\vv)^2]=\int_{M \times M}\rho(x,y)^2p_\vv(x,y)\,\mu(\d x)\mu(\d y).
\end{equation*}
Let $\alpha\in[0,1]$. For every
$\vv\in(0,1]$, define
\begin{equation*}\label{BE}
\eta^\aa(\vv):=1+\int_\vv^1 \gg(u)u^\aa\,\d u. 
\end{equation*}
Since $(0,\infty)\ni t\mapsto p_t(x,x)$ is decreasing for every $x\in M$ (see \eqref{A1} below),
it is clear that $\gg(t)$ is decreasing in $t$.  Under \eqref{DHU},  it is easy to see that, $\eta^\aa(\vv)<\infty$, $\vv\in(0,1]$.

Let $B\in\mathbf{B}$. Recall  that the semigroup $(P_t^B)_{t>0}$ is said to be  ultra-contractive if
\begin{equation}\label{UC}
\|P_t^B\|_{L^1(\mu)\to L^\infty(\mu)}:=\sup_{\|f\|_{L^1(\mu)}\le 1}\|P_t^B f\|_{L^\infty(\mu)}<\infty,\quad t>0.
\end{equation}
A classic result asserts that, \eqref{UC} is equivalent to that $(P_t^B)_{t>0}$ has a heat kernel $(p_t^B)_{t>0}$ and
$$\|P_t^B\|_{L^1(\mu)\to L^\infty(\mu)}=\sup_{x,y\in M}p_t^B(x,y)<\infty,\quad t>0;$$
see e.g. \cite{CKS87} or \cite[Section 2.1]{Davies89}. It is clear that, in the particular case when $B(t)=t$, \eqref{UC} is read as the ultra-contractivity of $(P_t)_{t>0}$, which implies \eqref{DHU} since $\mu$ is a finite measure.
\begin{remark}\label{rk-UC}
Let $B\in \mathbf{B}$.
\begin{itemize}
\item[(i)] It is proved in \cite[page 15]{WangWu} that, if $B$ satisfies \eqref{B-int} and there exists a constant $c>0$ such that $\|P_t\|_{L^1(\mu)\to L^\infty(\mu)}\leq c(1+t^{-d/2})$ for every $t>0$, then $(P_t^B)_{t>0}$ is ultra-contractive.

\item[(ii)] In general, the ultra-contractivity of $(P_t^B)_{t>0}$ implies additional regularity of $B$. For instance, if \eqref{UC} holds and there exists a constant $c>0$ such that, for some $x\in M$,
\begin{equation}\label{DL}
p_t(x,x)\geq \frac{c}{t^{d/2}},\quad t>0,
\end{equation}
then $B$ satisfies \eqref{B-int}. Indeed, by the mutual independence of $(X_t)_{t\geq0}$ and $(S_t^B)_{t\geq0}$, Fubini's theorem and \eqref{frac-eq} below,
\begin{equation*}\begin{split}
\infty>p^B_t(x,x)&=\int_{[0,\infty)}p_s(x,x)\,\mathbb{P}_{S_t^B}(\d s)\geq c\mathbb{E}\big[(S_t^B)^{-d/2}\big]\\
&=\frac{c}{\Gamma(\frac{d}{2})}\E\Big(\int_0^\infty \e^{-r S_t^B}r^{d/2-1}\,\d r\Big)\\
&=\frac{c}{\Gamma(\frac{d}{2})}\int_0^\infty \e^{-t B(r)}r^{d/2-1}\,\d r,
\quad t>0.
\end{split}\end{equation*}
We are aware that \eqref{DL} can be guaranteed by \cite[Theorem 7.2]{CouGri}. More precisely, if there exists a constant $C>0$ such that, for some $x\in M$,
$$\mu\big(B(x,2r)\big)\leq C\mu\big(B(x,r)\big),\quad r>0,$$
and
$$p_t(x,x)\leq\frac{C}{t^{d/2}},\quad t>0,$$
then \eqref{DL} holds.
\end{itemize}
\end{remark}

For every number $a>0$, let $\Pp_a$ be the subclass of $\Pp$ such that each element of $\Pp_a$ is absolutely continuous w.r.t. $\mu$ with the Radon--Nikodym derivative bounded by $a$; more precisely,
\begin{equation}\label{Pp_a}
\Pp_a:=\{\nu\in\Pp: \nu=h_\nu\mu,\,\|h_\nu\|_\infty\le a\}.
\end{equation}
Here and in the sequel, $\|\cdot\|_\infty$ denotes the supremum norm.

Now we begin to present the main results of the paper. The first one is on upper bound estimates.
\begin{theorem}\label{TH2.1}
 Assume that \eqref{DHU} holds.
\begin{itemize}
\item[\textnormal{(i)}] Let $B\in\mathbf{B}$. Then, for every $k\ge 1$,
\begin{equation}\label{UB1}
\limsup_{t\to\infty}\Big\{t\sup_{\nu\in \Pp_k}\E^\nu[\W_2(\mu_t^B,\mu)^2]\Big\}\le \sum_{i=1}^\infty\ff 8 {\ll_iB(\ll_i)}.
\end{equation}
\item[\textnormal{(ii)}] Let $B\in \mathbf{B}^\aa$ for some $\aa\in[0,1]$. Then, there exists a constant $c>0$ such that
\begin{equation}\label{UB2}
\sup_{\nu\in\Pp_k}\E^\nu[\W_2(\mu_t^B,\mu)^2]\le ck\inf_{\vv\in(0,1]}\big[\dd(\vv)+t^{-1}\eta^\aa(\vv)\big],\quad t,k\ge 1.
\end{equation}
\item[\textnormal{(iii)}] Let $B\in\mathbf{B}$.  If $(P_t^B)_{t>0}$ is ultra-contractive, then for every $\nu\in\Pp$ satisfying
\begin{equation}\label{CUUB1}
\int_0^1\!\! \int_M\E^\nu[\rho(x,X_t^B)^2]\,\mu(\d x) \,\d t <\infty,
\end{equation}
it holds that
\begin{equation}\label{UUB1}
\limsup_{t\to\infty}\left\{t\E^\nu[\W_2(\mu_t^B,\mu)^2]\right\}\le\sum_{i=1}^\infty\ff 8 {\ll_i B(\ll_i)}.
\end{equation}
\item[\textnormal{(iv)}] Let $B\in \mathbf{B}^\aa$ for some $\aa\in[0,1]$. If $(P_t^B)_{t>0}$ is ultra-contractive, then there exists a constant $C>0$ such that, for every $\nu\in \Pp$ and for all $t\ge 1$,
\begin{equation}\label{UUB2}
\E^\nu[\W_2(\mu_t^B,\mu)^2]\le C\left\{\ff 1 t \int_0^1 \E^\nu\left[\mu\big(\rho(X_s^B,\cdot)^2\big)\right]\,\d s+\inf_{\vv\in(0,1]}\big[\dd(\vv)+t^{-1}\eta^\aa(\vv)\big]\right\}.
\end{equation}
\end{itemize}
\end{theorem}
\begin{remark}\label{rk-thm1}
In Theorem \ref{TH2.1}, we do not make any assumptions on the boundary of $M$.
\end{remark}

A corollary of Theorem \ref{TH2.1}, which demands some curvature condition, is presented next. We consider the empty boundary case, which is enough for our purpose. For every $f\in C^2(M)$, denote the Hessian of $f$ by ${\rm Hess}_f$. Let ${\rm Ric}$ be the Ricci curvature of $M$.
\begin{corollary}\label{COR2.2}
Let $K_1,K_2\geq0$, $\aa\in[0,1]$ and $B\in \mathbf{B}^\aa$. Suppose that $\partial M$ is empty. Let $V=V_1+V_2$ for some functions $V_1\in C^2(M)$ and $V_2\in C^1(M)$  such that
\begin{equation}\label{Ric}
{\rm Ric}\ge -K_1, \quad {\rm Hess}_{V_1}\leq K_1, \quad \|\nn V_2\|_\infty\le K_2.
\end{equation} For any $t,\vv>0$, let
$$\tilde{\gg}(t):=\int_M\ff{\mu(\d x)}{\mu\big(B(x,\sqrt{t})\big)},\quad\quad \tilde{\eta}^\aa(\vv):=1+\int_\vv^1\tilde{\gg}(u)u^\aa\,\d u.$$
If \eqref{DHU} holds, then there exists a constant $c>0$ such that,  for every $t,k\ge 1$,
\begin{equation*}\label{UB3}
\sup_{\nu\in\Pp_k}\E^\nu[\W_2(\mu_t^B,\mu)^2]\le ck\inf_{\vv\in(0,1]}\left\{\big[1+\mu(|\nabla V|^2)\big]\vv + \frac{1}{t}\tilde{\eta}^\aa(\vv)\right\}.
\end{equation*}
\end{corollary}
\begin{remark}\label{rk-cor}
Note that, in Corollary \ref{COR2.2}, $\mu(|\nabla V|^2)=\infty$ is allowed. Note also that, it is not necessary to require that $\tilde{\gg}(t)$ is finite for every $t>0$. However, if in addition we have the on-diagonal lower bound estimate of the heat kernel, i.e.,
$$p_t(x,x)\geq\frac{a_t}{\mu\big(B(x,\sqrt{t})\big)},\quad x\in M,\,t>0,$$
for some function $a:(0,\infty)\rightarrow(0,\infty)$, then \eqref{DHU} is equivalent to $\tilde{\gg}(t)<\infty$ for all $t>0$. Indeed, it is clear from the former to the later, and the inverse implication is deduced from \cite[Lemma 2.3]{GW2001} which implies that, under assumption \eqref{Ric}, there exists some constant $c>0$ such that
$$p_t(x,x)\leq\frac{c\e^{ct}}{\mu\big(B(x,\sqrt{t})\big)},\quad x\in M,\,t>0.$$
\end{remark}

In order to present results on lower bound estimates, we should introduce the truncated $L^p$-Wasserstein distance as \cite{Wang20202}. Let $p\in(0,\infty)$. For every $\nu_1,\nu_2\in\Pp$, let
$$\tilde{\W}_p(\nu_1,\nu_2):=\inf_{\pi\in\C(\nu_1,\nu_2)}\left(\int_{M\times M}\big[1\wedge \rho(x,y)\big]^p\,\pi(\d x,\d y)\right)^{\frac{1}{p\vee 1}}.$$
Obviously, $\tilde{\W}_1\le\W_1\le\W_2$ as functions on $\Pp\times\Pp$.  Recall that the boundary $\partial M$ is said to be convex if the second fundamental form is nonnegative definite.
\begin{theorem}\label{TH3.1}
\textnormal{(1)} Let $B\in \mathbf{B}$. Then there exists a constant $c>0$ such that
\begin{equation}\label{LB1}
\E^\mu[\tilde{\W}_1(\mu_t^B,\mu)]\ge c t^{-\frac{1}{2}},\quad t\ge 1.
\end{equation}
Moreover, if \eqref{SG} holds, then for every $\nu\in\Pp$,
\begin{equation}\label{LB2}
\liminf_{t\to\infty}\big\{t\E^{\nu}[\tilde{\W}_1(\mu_t^B,\mu)^2]\big\}>0.
\end{equation}
\textnormal{(2)} Assume that $\pp M$ is either empty or convex.
\begin{itemize}
\item[(i)] Let $B\in \mathbf{B}_\aa$ for some $\aa\in[0,1]$ and $0<p<\aa$.  Suppose that $\mu(|\nn V|)<\infty$ and
\begin{equation}\label{LRIC}
\textnormal{Ric}\ge -K,\quad V\le K,
\end{equation}
for some constant $K>0$. If $d> 2\aa$, then there exists a constant $c>0$ such that
\begin{equation}\label{LB3}
\inf_{\nu\in\Pp_k}\E^\nu[\tilde{\W}_1(\mu_t^B,\mu)]\ge\Big\{\inf_{\nu\in\Pp_k}\E^\nu[\tilde{\W}_p(\mu_t^B,\mu)]\Big\}^{\ff 1 p}
\ge ck^{-\ff {2\aa}{p(d-2\aa)}}t^{-\ff 1{d-2\aa}},\quad k,t\ge 1.
\end{equation}
\item[(ii)] Let $B\in \mathbf{B}$ and $V\in C^2(M)$. Suppose that $(P_t)_{t>0}$ is ultra-contractive and $\textnormal{Ric}-\textnormal{Hess}_V\ge K$ for some constant $K\in\R$. Then
\begin{equation}\label{LB4}
\liminf_{t\to\infty}\inf_{\nu\in\Pp}\big\{t\E^\nu[\W_2(\mu_t^B,\mu)^2]\big\}\ge \sum_{i=1}^\infty \ff 2 {\ll_i B(\ll_i)}.
\end{equation}
\end{itemize}
\end{theorem}

We point out that, when $B(t)=t$ for all $t\geq0$, the main results above go back to the situation in \cite{Wang20202}. However, it is remarked that,  in this particular case, even by further assuming $\partial M$ is empty and $\textnormal{Ric}-\textnormal{Hess}_V\ge K$ for some constant $K\in\R$, it is difficult to prove that
$$\limsup_{t\to\infty}\left\{t\E^\nu[\W_2(\mu_t^B,\mu)^2]\right\}\le \sum_{i=1}^\infty\ff 2 {\ll_iB(\ll_i)},\quad\nu\in\Pp.$$
See Remark 1.1 in the aforementioned paper for more details.

We should mention that the fundamental idea of proofs for the main results is motivated by recent works \cite{Wang20202} and \cite{WangWu}. However, in order to deal with the present case, we need to develop some new techniques.

The Riemannian structure is a convenient setting for this work; however, the approach covers more general situations. We end this section with the following remark on potential extensions of the above results to metric measure spaces.
\begin{remark}
Let $K\in\R$ and $N\in(1,\infty)$. By the same approach presented in Sections 4 and 5 below, results in Theorem \ref{TH2.1} and Theorem \ref{TH3.1} can be established similarly on a large class of  not necessarily smooth metric probability measure spaces, namely, ${\rm RCD}^\ast(K,N)$ spaces with  the reference measure being a probability measure. See e.g. \cite{AGS2014,Gigli2015,EKS2015,JLZ2016,Li2013} for the definition of the ${\rm RCD}^\ast(K,N)$ space and necessary details needed, e.g. properties on the heat flow. As for our present setting, if $\partial M$ is empty or convex, and $V$ belongs to $C^2(M)$ such that $L|\nabla f|^2-2\langle\nabla L f,\nabla f\rangle\geq 2K|\nabla f|^2+2(L f)^2/N$ for every $f\in C^\infty(M)$, then $(M,\rho,\mu)$ is an ${\rm RCD}^\ast(K,N)$ space.
\end{remark}

\section{Example}\hskip\parindent
In order to illustrate the results in Theorems \ref{TH2.1} and \ref{TH3.1},  it is necessary to give the following concrete example on $\R^d$; see \cite[Example 1.4]{Wang20202}.
\begin{example}[$M=\R^d$]\label{EXP4.1}
Let $\kappa>0$ and $q>1$. Consider $V(x)=-\kappa|x|^q+U(x)$, $x\in\R^d$, for some function  $U\in C^1(\R^d)$
with $\|\nn U\|_\infty<\infty$.
\begin{itemize}
\item[(1)] Let $B\in \mathbf{B}^\aa$ for some $\aa\in[0,1]$. Then there exists a constant $c>0$ such that, for any $t,k\ge 1$,
\begin{equation}\label{EXPUB1}
\sup_{\nu\in\Pp_k}\E^\nu[\W_2(\mu_t^B,\mu)^2]\le
\begin{cases}
ckt^{-\ff {2(q-1)}{(d-2\aa)q+2\aa}},\quad &\mbox{if}~2(1+\aa)(q-1)<dq,\\
ckt^{-1}\log(1+t),\quad  &\mbox{if}~2(1+\aa)(q-1)=dq,\\
ckt^{-1},\quad  &\mbox{if}~2(1+\aa)(q-1)>dq.\\
\end{cases}
\end{equation}

\item[(2)] Let $B\in \mathbf{B}_\aa$ for some $\aa\in[0,1]$. For any $\nu\in\Pp$, there exists a constant $c>0$ such that, for large enough $t>0$,
\begin{equation*}\label{EXPLB2}\E^\nu[\W_2(\mu_t^B,\mu)^2]\ge\E^\nu[\tilde{\W}_1(\mu_t^B,\mu)^2]\ge ct^{-\ff 2 {2\vee (d-2\aa)}}.\end{equation*}
\end{itemize}
\end{example}

We remark that, results in Example \ref{EXP4.1} are sharp in the following sense. Let $\nu\in\Pp_k$. If $dq<2(1+\alpha)(q-1)$, then both the upper and the lower bounds of $\E^\nu[\W_2(\mu_t^B,\mu)^2]$ behave as $t^{-1}$. If $dq>2(1+\alpha)(q-1)$, then the lower bound of $\E^\nu[\W_2(\mu_t^B,\mu)^2]$ behaves as $t^{-\frac{2}{d-2\alpha}}$, and the upper bound of $\E^\nu[\W_2(\mu_t^B,\mu)^2]$ behaves as $t^{-\ff {2(q-1)}{(d-2\aa)q+2\aa}}$ which converges to $t^{-\frac{2}{d-2\alpha}}$ as $q\rightarrow\infty$. Refer to \cite{WangWu} for sharp results in the case when $M$ is compact and $B$ satisfies \eqref{B-int} in addition.

\section{Proofs of Theorem \ref{TH2.1} and Corollary \ref{COR2.2}}\hskip\parindent
In this section, we aim to prove Theorem \ref{TH2.1} and Corollary \ref{COR2.2}. At first, let us give a brief description of the  idea of
proof.  Let $t,\varepsilon>0$.  The key step is to construct a regularized (or smoothed) version of $\mu_t^B$, denoted by $\mu_{\vv,t}^B$, such that it is extremely close to $\mu_t^B$  under the  Wasserstein distance as $\vv$ is small enough. Indeed, the regularized measure $\mu_{\vv,t}^B$ is obtained from $\mu$ through the heat flow, namely, $\mu_{\vv,t}^B=\mu_t^B P_\vv$. Then, to estimate $\E[\W_2(\mu_t^B,\mu)^2]$, by the triangle inequality, we have
\begin{equation}\label{TRI}
\E[\W_2(\mu_t^B,\mu)^2]\le(1+\delta)\E[\W_2(\mu_{\vv,t}^B,\mu)^2]+(1+\delta^{-1})
\E[\W_2(\mu_t^B,\mu_{\vv,t}^B)^2],\quad\delta>0.
\end{equation}
So we only need to investigate the two terms on the right-hand of \eqref{TRI}. It is worth  pointing out that the main purpose of regularization
is to employ the following inequality proved in \cite[Theorem 2]{Ledoux1} (see also \cite[Proposition 2.3]{AST2019} or \cite[Theorem A.1]{Wang20202}), i.e.,
\begin{equation}\label{WUP}
\W_2(f\mu,\mu)^2\le 4\mu\big(|\nn(-L)^{-1}(f-1)|^2\big),\quad f\ge 0,\,\mu(f)=0,\,\mu\in\mathscr{P}.
\end{equation}
Then, by \eqref{WUP}, we can get the upper estimate of $\E[\W_2(\mu_{\vv,t}^B,\mu)^2]$. As for $\E[\W_2(\mu_t^B,\mu_{\vv,t}^B)^2]$,
we use an approximation strategy.

Let $\{\phi_i\}_{i\in\mathbb{N}}$ be the sequence of orthonormal eigenfunctions corresponding to $\{\ll_i\}_{i\in\mathbb{N}}$ (satisfying the Neumann boundary condition if $\partial M\neq\emptyset$). It is well known that the heat kernel $(p_t)_{t>0}$ of the diffusion semigroup $(P_t)_{t>0}$ has the following representation formula, i.e.,
\begin{equation}\label{HK}
p_t(x,y)=1+\sum_{i=1}^\infty \e^{-\ll_i t}\phi_i(x)\phi_i(y),\quad t>0,\,x,y\in M.
\end{equation}
Let $\vv,t>0$. We define
\begin{equation*}\label{NR}
f_{\vv,t}^B:=\ff{\d \mu_{\vv,t}^B}{\d \mu}.
\end{equation*}
Then, letting
$$\xi_i^B(t):=\ff 1 t\int_0^t\phi_i(X_s^B)\,\d s,$$
by \eqref{HK}, we have
\begin{equation}\label{NRSR}
f_{\vv,t}^B=\ff 1 t \int_0^t p_\vv(X_s^B,\cdot)\,\d s=1+\sum_{i=1}^\infty\e^{-\ll_i \vv}\xi_i^B(t)\phi_i.
\end{equation}

We should point out that, since $(P_t)_{t\geq0}$ is $\mu$-invariant, so is $(P_t^B)_{t\geq0}$; see e.g. \cite{SSV2012}. Indeed, for every $f\in L^1(\mu)$, by the mutual independence of $(X_t)_{t\geq0}$ and $(S_t^B)_{t\geq0}$ and Fubini's theorem,
\begin{equation}\begin{split}\label{invariant}
\mu(P_t^B f)&=\int_M \Big(\int_{[0,\infty)} P_sf \,\P_{S_t^B}(\d s)\Big)\,\d\mu=\int_{[0,\infty)} \mu(P_sf) \,\P_{S_t^B}(\d s)\\
&=\int_{[0,\infty)} \mu(f) \,\P_{S_t^B}(\d s)=\mu(f),\quad t\geq0.
\end{split}\end{equation}

Now we are ready to prove Theorem \ref{TH2.1}.
\begin{proof}[Proof of Theorem \ref{TH2.1}]
We divide the proof for \eqref{UB1}, \eqref{UB2}, \eqref{UUB1} and \eqref{UUB2} into four parts.

(1) We may assume that $\sum_{i=1}^\infty\big(\ll_i B(\ll_i)\big)^{-1}<\infty$. By \cite[(2.6)]{WangWu},
 which still holds in the present noncompact setting by a slight modification of the original proof using assumption \eqref{DHU},
 there exists a constant $c>0$ such that, for every $t\geq1$ and every $\varepsilon>0$,
$$\sup_{\nu\in\Pp_k}\Big|t\E^\nu\big[\mu\big(|\nn(-L)^{-1}(f_{\vv,t}^B-1)|^2\big)\big]-\sum_{i=1}^\infty\ff 2{\ll_i B(\ll_i)\e^{2\vv\ll_i}}\Big|
\le\ff{ck}{t}\sum_{i=1}^\infty\ff 1 {\ll_i B(\ll_i)\e^{2\vv\ll_i}}.$$
Then, applying \eqref{WUP}  with $f_{\vv,t}^B$ instead of $f$, we immediately  obtain the estimate
\begin{equation}\label{TH2.1(1)}
t\sup_{\nu\in \Pp_k}\E^\nu[\W_2(\mu_{\vv,t}^B,\mu)^2]\le \sum_{i=1}^\infty\ff 8 {\ll_i B(\ll_i)}+\ff{ck}{t}\sum_{i=1}^\infty\ff{4}{\ll_i B(\ll_i)},\quad t\geq1,\,\vv>0.
\end{equation}

Let $n\in\mathbb{N}$. To estimate the error term $\E^\nu[\W_2(\mu_t^B,\mu_{\vv,t}^B)^2]$, we consider
the truncated Wasserstein distance
$$\W_{2,n}(\mu_1,\mu_2):=\inf_{\pi\in \C(\mu_1,\mu_2)}\left(\int_{M\times M}\big[n\wedge \rho(x,y)^2\big]\,\pi(\d x,\d y)\right)^{\ff 1 2},\quad \mu_1,\mu_2\in\Pp.  $$
For every $t>0$, since $(\mu_t^B P_\varepsilon)_{\varepsilon>0}$ converges weakly to $\mu_t^B$ as $\varepsilon\downarrow0$, we have
$$\limsup_{\vv\downarrow 0}\W_{2,n}(\mu_{\vv,t}^B,\mu_t^B)^2=0; $$
see also the proof of \cite[Theorem 1.1]{Wang20202} on page 276.
Combining this with the fact that $\W_{2,n}(\mu_{\vv,t}^B,\mu_t^B)\le n$ and $\E^\nu(\cdot)\leq k\E^\mu(\cdot)$ for every $\nu\in\Pp_k$, by Fatou's lemma,  we have, for every $t>0$,
$$\limsup_{\vv\downarrow 0}\sup_{\nu\in \Pp_k}\E^\nu[\W_{2,n}(\mu_{\vv,t}^B,\mu_t^B)^2]
\le k \limsup_{\vv\downarrow 0}\E^\mu[\W_{2,n}(\mu_{\vv,t}^B,\mu_t^B)^2]\leq0.$$
By \eqref{TH2.1(1)}, the triangle inequality for $\W_{2,n}$ and Fatou's lemma, we derive that, for every $\nu\in\Pp_k$,
\begin{equation*}\begin{split}\label{TH2.1(2)}
t\E^\nu[\W_{2,n}(\mu_t^B,\mu)^2]&\le t\limsup_{\vv\downarrow 0}\E^\nu\big[\W_{2,n}(\mu_t^B,\mu_{\vv,t}^B)+\W_{2,n}(\mu_{\vv,t}^B,\mu)\big]^2\\
&\le\sum_{i=1}^\infty\ff 8 {\ll_i B(\ll_i)}+\ff{ck}{t}\sum_{i=1}^\infty\ff 4 {\ll_i B(\ll_i)},
\quad  t\geq1,
\end{split}\end{equation*}
or
\begin{equation}\begin{split}\label{TH2.1(2)}
t\sup_{\nu\in\Pp_k}\E^\nu[\W_{2,n}(\mu_t^B,\mu)^2]\le \sum_{i=1}^\infty\ff 8 {\ll_i B(\ll_i)}+\ff{ck}{t}\sum_{i=1}^\infty\ff 4 {\ll_i B(\ll_i)},
\quad t\geq1.
\end{split}\end{equation}

Thus, by the monotone convergence theorem, we arrive at
\begin{eqnarray*}
&&t\sup_{\nu\in\Pp_k}\E^\nu[\W_{2}(\mu_t^B,\mu)^2]=t\sup_{\nu\in\Pp_k}\E^\nu[\sup_{n\geq1}\W_{2,n}(\mu_t^B,\mu)^2]\\
&=&t\sup_{\nu\in\Pp_k}\sup_{n\geq1}\E^\nu[\W_{2,n}(\mu_t^B,\mu)^2]=\sup_{n\geq1}\big\{t\sup_{\nu\in\Pp_k}\E^\nu[\W_{2,n}(\mu_t^B,\mu)^2]\big\}\\
&\leq&\sum_{i=1}^\infty\ff 8 {\ll_i B(\ll_i)}+\ff{ck}{t}\sum_{i=1}^\infty\ff 4 {\ll_i B(\ll_i)}<\infty,\quad t\geq1.
\end{eqnarray*}
Taking $t\rightarrow\infty$,  we immediately obtain the desired result \eqref{UB1}.

(2) Note that, for every $h\in L^2(\mu)$ with $\mu(h)=0$,
$$(-L)^{-1}h=\int_0^\infty P_s h\,\d s,$$
which clearly belongs to $\mathcal{D}(L)$. Then the integration-by-parts formula and the symmetry of $(P_s)_{s\geq0}$ in $L^2(\mu)$ lead to
$$\int_{M}\big|\nn (-L)^{-1}(f_{\vv,t}^B-1)\big|^2\,\d \mu=\int_{0}^\infty\d s\int_M\big|P_{\ff s 2}
f_{\vv,t}^B-1\big|^2\,\d \mu,
\quad t,\varepsilon>0.  $$
Combining \eqref{WUP}, \eqref{HK} and \eqref{NRSR}, by a careful calculation, we obtain
\begin{equation}\label{TH2.1(4)}
\W_2(\mu_{\vv,t}^B,\mu)^2\le 4\sum_{i=1}^\infty\frac{|\xi_i^B(t)|^2}{\ll_i\e^{2\ll_i\vv}},\quad t,\varepsilon>0.
\end{equation}

To show $\eqref{UB2}$, noting again that $\E^\nu(\cdot)\le k\E^\mu(\cdot)$ for every $\nu\in\Pp_k$, it is enough for us to prove that, there exists some constant $c>0$ such that
\begin{eqnarray}\label{UB2'}
\E^\mu[\W_2(\mu_t^B,\mu)^2]\le c\inf_{\vv\in(0,1]}\big[\dd(\vv)+t^{-1}\eta^\aa(\vv)\big],\quad t\geq1.
\end{eqnarray}

Due to  \eqref{TH2.1(4)}, in order to bound $\E^\mu[\W_2(\mu_{\vv,t}^B,\mu)^2]$, it is crucial to estimate $\E^\mu[|\xi_i^B(t)|^2]$.
Let $i\in\mathbb{N}$. One may use the fact that $(P_t^B)_{t\geq0}$ is $\mu$-invariant (see \eqref{invariant} above) and $\mu(\phi_i^2)=1$
to get the identity
\begin{equation*}\label{B-invant}
\E^\mu[\phi_i(X_{s}^B)^2]=\mu(\phi_i^2)=1,\quad s\geq0.
\end{equation*}
Next, the Markov property yields
\begin{equation*}\label{mp}
\E^x[\phi_i(X_{s_2}^B)|X_{s_1}^B]=P_{s_2-s_1}^B\phi_i(X_{s_1}^B)
=\e^{-B(\ll_i)(s_2-s_1)}\phi_i(X_{s_1}^B),\quad s_2\ge s_1\geq0,\,x\in M.
\end{equation*}
Hence, we conclude from the preceding  identities and the definition of $\xi_i^B(t)$ that (see e.g. \cite[(3.12)]{WangWu})
\begin{equation}\label{xi}\begin{split}
\E^\mu\big[|\xi_i^B(t)|^2\big]&=\ff 2 {t^2}\int_0^t \d s_1\int_{s_1}^t\E^\mu\big[\phi_i(X_{s_1}^B)\phi_i(X_{s_2}^B)\big]\,\d s_2\\
&=\ff 2 {t^2}\int_0^t\d s_1 \int_{s_1}^t\E^\mu\big[\phi_i(X_{s_1}^B)^2\big]\e^{-B(\ll_i)(s_2-s_1)}\,\d s_2\\
&\le\ff 2 {tB(\ll_i)},\quad t>0.
\end{split}\end{equation}

Note that $B\in\mathbf{B}^\aa$ for some $\aa\in[0,1]$. It is an elementary fact that
\begin{equation}\label{LB}
B(t)\geq \kappa(t\wedge t^\alpha),\quad t\geq0,
\end{equation}
for some constant $\kappa>0$.
Combing \eqref{LB} with \eqref{TH2.1(4)} and \eqref{xi}, we deduce that, there exist constants $c_1,c_2>0$ such that
\begin{equation}\begin{split}\label{TH2.1(5)}
\E^\mu[\W_2(\mu_{\vv,t}^B,\mu)^2]&\le\ff 8 t \sum_{i=1}^\infty \ff 1 {\ll_iB(\ll_i)}\e^{-2\ll_i\vv}
\le \ff {c_1} t \sum_{i=1}^\infty \ff 1 {\ll_i^{1+\aa}}\e^{-2\ll_i\vv}\\
&=\ff {c_1} t \ff 1 {\Gamma(\aa)}\int_{0}^\infty \sum_{i=1}^\infty\ff 1 {\ll_i}\e^{-(s+2\vv)\ll_i}s^{\aa-1}\,\d s\\
&\le\ff {c_2} t \sum_{i=1}^\infty\int_0^\infty\Big(\int_{\ff {s+2\vv} 2}^\infty \e^{-2\ll_i u}s^{\aa-1}\,\d u\Big)\, \d s,
\quad t,\varepsilon>0,
\end{split}\end{equation}
where the equality is due to the fact that
\begin{equation}\label{frac-eq}
\ff 1 {\ll^\aa}=\ff 1 {\Gamma(\aa)}\int_0^\infty \e^{-s\ll}s^{\aa-1}\,\d s,\quad\alpha,\lambda>0,
\end{equation}
and $\Gamma(\cdot)$ denotes the Gamma function. Combining \eqref{SG} with properties of the heat kernel $(p_t)_{t>0}$, we have
\begin{equation*}\begin{split}
p_{2t}(x,x)-1&=\int_M|p_t(x,y)-1|^2\,\mu(\d y)=\int_{M}|P_{\ff t 2}[p_{\ff t 2}(x,\cdot)](y)-1|^2\,\mu(\d y)\\
&\le\e^{-\ll_1 t}\int_M|p_{\ff t 2}(x,y)-1|^2\,\mu(\d y)=\e^{-\ll_1 t}[p_t(x,x)-1],\quad t>0.
\end{split}\end{equation*}
Then, together with  \eqref{HK}, we obtain
\begin{equation*}\begin{split}
\sum_{i=1}^\infty\e^{-2\ll_i t}&=\int_M[p_{2t}(x,x)-1]\,\mu(\d x)\\
&\le\e^{-\ll_1 t}\int_M [p_t(x,x)-1]\,\mu(\d x)\\
&\le\e^{-\ll_1 t}\gg(t),\quad t>0.
\end{split}\end{equation*}
Substituting this into \eqref{TH2.1(5)}, we deduce that,  there exists a constant $c_3>0$ such that, for ever $t>0$ and every $\vv\in(0,1]$,
\begin{equation}\begin{split}\label{TH2.1(6)}
\E^\mu[\W_2(\mu_{\vv,t}^B,\mu)^2]&\le\ff {c_2} t\int_0^\infty\Big(\int_{\ff{s+2\vv}2}^\infty\e^{-\ll_1 u}\gg(u)s^{\aa-1}\,\d u\Big)\,\d s\\
&=\ff {c_2} t  \int_\vv^\infty\Big(\int_0^{2u-2\vv}\e^{-\ll_1 u}\gg(u)s^{\aa-1}\,\d s\Big)\,\d u\\
&=\ff {c_2} t \int_\vv^1\Big(\int_0^{2u-2\vv}\e^{-\ll_1 u}\gg(u)s^{\aa-1}\,\d s\Big)\,\d u\\
&\quad+\ff {c_2} t \int_1^\infty\Big(\int_0^{2u-2\vv}\e^{-\ll_1 u}\gg(u)s^{\aa-1}\,\d s\Big)\,\d u\\
&\le\ff {c_2} t \int_\vv^1\gg(u)\Big(\int_0^{2u}s^{\aa-1}\,\d s\Big)\,\d u\\
&\quad+\ff{c_2} t \int_1^\infty\e^{-\ll_1 u}\gamma(1)\Big(\int_0^{2u}s^{\aa-1}\,\d s\Big)\,\d u\\
&\le \ff {c_3} t \left(1+\int_\vv^1 \gg(u)u^\aa \,\d u\right)=\ff {c_3} t\eta^\aa(\vv).
\end{split}\end{equation}

On the other hand, it is easy to verify that, for every $t,\varepsilon>0$,
$$\pi(\d x,\d y):=\ff 1 t\int_0^t \big\{\delta_{X_s^B}(\d x)\,p_{\vv}(X_s^B,y)\,\mu(\d y)\big\}\,\d s\in\C(\mu_t^B,\mu_{\vv,t}^B).$$
Then, by this and  the $\mu$-invariance of $(P_t^B)_{t\geq0}$ again,  we have
\begin{equation}\label{TH2.1(6+)}\begin{split}
\E^{\mu}[\W_2(\mu_t^B,\mu_{\vv,t}^B)^2]&\le\ff 1 t\E^\mu\left[\int_0^t\d s\int_M\rho(X_s^B,y)^2 p_\vv(X_s^B,y)\,\mu(\d y)\right]\\
&=\ff 1 t\int_0^t\mu\left[P_s^B\left(\int_M p_\vv(\cdot,y)\rho(\cdot,y)^2\,\mu(\d y)\right)\right]\,\d s\\
&=\E^\mu[\rho(X_0,X_\vv)^2]=\dd(\vv),\quad t,\varepsilon>0.
\end{split}\end{equation}

By the triangle inequality for $\W_2$ (see \eqref{TRI}), we thus conclude from  \eqref{TH2.1(6)} and \eqref{TH2.1(6+)} that
$$\E^\mu[\W_2(\mu_t^B,\mu)^2]\le 2\inf_{\vv\in(0,1]}\big[\dd(\vv)+c_3t^{-1}\eta^\aa(\vv)\big].$$
We complete the proof of \eqref{UB2'}, and hence \eqref{UB2}.

(3) Now we turn to prove \eqref{UUB1}. Without loss of generality, assume that $\sum_{i=1}^\infty\big(\ll_i B(\ll_i)\big)^{-1}<\infty$. Let $\vv>0$. By the ultra-contractivity of $(P_t^B)_{t>0}$,
it is obvious to see that
$$\zeta^B(\vv):=\sup_{t\ge \vv,\,x,y\in M}p_t^B(x,y)<\infty.$$
Let $\nu\in\Pp$ be the initial distribution of $(X_t^B)_{t\geq0}$, and let $\nu_\vv^B$ denote the distribution of $X_\vv^B$. Then $\nu\in \Pp_{\zeta^B(\vv)}$ (see \eqref{Pp_a} above for its definition). Let
$$\bar{\mu}_{\vv,t}^B:=\ff 1 t \int_{\vv}^{t+\vv}\delta_{X_s^B}\,\d s,\quad t>0.$$
It is standard to deduce from the Markov property that
$$\E^\nu[\W_2(\bar{\mu}_{\vv,t}^B,\mu)^2]=\E^{\nu_\vv^B}[\W_2(\mu_t^B,\mu)^2],\quad t>0.$$
Then, by  \eqref{TH2.1(2)}, we have
\begin{equation}\label{TH2.1(3)}
\limsup_{t\to\infty}\big\{t\E^\nu[\W_2(\bar{\mu}_{\vv,t}^B,\mu)^2]\big\}=\limsup_{t\to\infty}
\big\{t\E^{\nu_\vv^B}[\W_2(\mu_t^B,\mu)^2]\big\}\le\sum_{i=1}^\infty\ff 8 {\ll_i B(\ll_i)}.
\end{equation}
By the definition of coupling, for every  $t\geq\varepsilon$,
$$\pi_0:=\ff 1 t\int_0^\vv \dd_{(X_s^B,X_{s+t}^B)}\,\d s+\ff 1 t\int_{\vv}^t \dd_
{(X_s^B,X_s^B)}\,\d s\in\C(\mu_t^B,\bar{\mu}_{\vv,t}^B).$$
Since the conditional distribution of $X_{s+t}^B$ given $X_s^B$ is bounded by
$\zeta^B(1)\mu$ for every $t\ge 1$ and $s\geq0$, by Fubini's theorem,  we have
\begin{equation*}\begin{split}
t\E^\nu[\W_2(\mu_t^B,\bar{\mu}_{\vv,t}^B)^2]&\le t\E^\nu\Big[\int_{M\times M}\rho(x,y)^2\,\pi_0(\d x,\d y)\Big]\\
&=\int_0^\vv\E^\nu\big[\rho(X_s^B,X_{s+t}^B)^2\big]\,\d s\\
&\le \zeta^B(1)\int_0^\vv \E^\nu\big[\mu\big(\rho(X_s^B,\cdot)^2\big)\big]\,\d s=:r_\vv^B,\quad t\geq1\vee\varepsilon.
\end{split}\end{equation*}
Combining this with \eqref{CUUB1}, \eqref{TH2.1(3)} and \eqref{TRI}, due to the fact that $\lim_{\varepsilon\downarrow0}r_\vv^B=0$, we derive
\begin{equation*}\begin{split}
&\limsup_{t\to\infty}\big\{t\E^\nu[\W_2(\mu_t^B,\mu)^2]\big\}\\
&\le\lim_{\vv\downarrow 0}\Big(\big[1+(r_\vv^B)^{\ff 1 2}\big]\limsup_{t\to\infty}\big\{t\E^\nu[\W_2(\bar{\mu}_{\vv,t}^B,\mu)^2]\big\}
+\big[1+(r_{\vv}^B)^{-\ff 1 2}\big]r_\vv^B\Big)\\
&\le\sum_{i=1}^\infty\ff 8 {\ll_i B(\ll_i)},
\end{split}\end{equation*}
which finishes the proof of \eqref{UUB1}.

(4) We turn to prove \eqref{UUB2}. The proof is short and essentially the same as the one for \cite[(1.10)]{Wang20202}. We present it here for completeness. Since $(P_t^B)_{t>0}$ is ultra-contractive, it is clear that there exists a constant $c_4>0$ such that
\begin{equation}\label{TH2.1(7)}
\sup_{t\ge 1}p_t^B(x,y)\le c_4,\quad x,y\in M.
\end{equation}
Then $\ff{\d \nu_1^B}{\d \mu}\le c_4$, where $\nu_1^B$ is the distribution of $X_1^B$.  Let $t\geq1$. Define  $\bar{\mu}_t^B=\ff 1 t\int_0^t\dd_{X_{1+s}^B}\,\d s$. It is easy to see that
$$\pi:=\ff 1 t \int_0^1 \dd_{(X_s^B,X_{s+t}^B)}\,\d s+\ff 1 t \int_1^t \dd_{(X_s^B,X_s^B)}\,\d s\in\C(\mu_t^B,\bar{\mu}_t^B).$$
Applying \eqref{TH2.1(7)},  we obtain
\begin{equation*}\label{TH2.1(8)}\begin{split}
\E^\nu[\W_2(\mu_t^B,\bar{\mu}_t^B)^2]&\le\ff 1 t\E^\nu\Big[\int_0^1\rho(X_s^B,X_{s+t}^B)^2\,\d s\Big]\\
&\le\ff {c_4} t\E^\nu\Big[\int_0^1\mu\big(\rho(X_s^B,\cdot)^2\big)\,\d s\Big].
\end{split}\end{equation*}
By the Markov property and \eqref{UB2}, we can find a constant $c_5>0$ such that
$$\E^\nu[\W_2(\bar{\mu}_t^B,\mu)^2]=\E^{\nu_1^B}[\W_2(\mu_t^B,\mu)^2]\le c_5\inf_{\vv\in(0,1]}\big[\dd(\vv)+t^{-1}\eta^\aa(\vv)\big].$$
Thus, the triangle inequality for $\W_2$ leads to \eqref{UUB2} for some constant $C>0$.
\end{proof}

Now we turn to prove the corollary. The proof is based on \cite[Page 279]{Wang20202}. However, there is a gap in the original proof of \cite[(2.18)]{Wang20202} by careful study. So, it is necessary to give some details below to fix it.
\begin{proof}[Proof of Corollary \ref{COR2.2}] By \cite[(2.17)]{Wang20202} (which is derived from assumption \eqref{Ric}), we have
$$\gg(t)\le c_1 \tilde{\gg}(t),\quad t\in(0,2],$$
where $c_1>0$ is a constant. Hence, $\eta^\alpha(\varepsilon)\leq(1\vee c_1)\tilde{\eta}^\alpha(\varepsilon)$, $\varepsilon\in(0,1]$.

Let $\rho_x(\cdot)=\rho(x,\cdot)$, $x\in M$. Then, since ${\rm Ric}\geq -K_1$, the Laplacian comparison theorem (see e.g. \cite[Theorem 1.1.10]{Wang2014}) implies
$$\Delta \rho_x(y)^2\leq2d +2\sqrt{(d-1)K_1}\, \rho_x(y),\quad (x,y)\in\hat{M},$$
where $\hat{M}:=\{(x,y)\in M\times M: x\neq y,\,y\notin {\rm cut}(x)\}$ and ${\rm cut}(x)$ is the cut-locus of $x$. Then
\begin{equation*}\begin{split}
L \rho_x(y)^2&=\Delta\rho_x(y)^2+2\langle\nabla V(y),\nabla\rho_x(y)\rangle\rho_x(y)\\
&\leq c_2\big[ 1+|\nabla V|(y)\big]\rho_x(y),\quad (x,y)\in\hat{M},
\end{split}\end{equation*}
for some constant $c_2>0$, where we also used the fact that $|\nabla \rho_x|=1$ whenever $\rho_x(\cdot)$ is smooth. Combining this with It\^{o}'s formula for the radial process originally due to W.S. Kendall (see e.g. \cite[Chapter 2]{Wang2014}), we have
\begin{equation*}\begin{split}
&\d  \rho_{X_0}(X_t)^2 \leq \d\mathcal{N}_t + L\rho_{X_0}(X_t)^2\d t\\
&\leq \d\mathcal{N}_t+ c_3\big[1+|\nabla V|^2(X_t)\big]\d t  + c_3 \rho_{X_0}(X_t)^2 \d t,\quad t\in[0,1],
\end{split}\end{equation*}
for some constant $c_3>0$, where $(\mathcal{N}_t)_{t\geq0}$ is a martingale on some filtered probability space. Hence, by the $\mu$-invariance of $(P_t)_{t\geq0}$, we obtain
$$\E^\mu\big[\rho_{X_0}(X_t)^2\big]\leq c_4\big[1+\mu(|\nabla V|^2)\big]t+ c_4\int_0^t\E^\mu\big[\rho_{X_0}(X_s)^2\big]\,\d s,\quad t\in[0,1],$$
for some constant $c_4>0$. Thus, Gr\"{o}nwall's inequality immediately leads to
\begin{equation*}\begin{split}
\delta(\varepsilon)&=\E^\mu\big[\rho(X_0,X_\varepsilon)^2\big]\leq c_4 \big[1+\mu(|\nabla V|^2)\big]\varepsilon \e^{c_4 \varepsilon}\\
&\leq c_5\big[1+\mu(|\nabla V|^2)\big]\varepsilon,\quad \varepsilon\in[0,1],
\end{split}\end{equation*}
for some constant $c_5>0$.

Let $k\geq1$ and $\nu\in\Pp_k$. Since $\E^\nu(\cdot)\leq k\E^\mu(\cdot)$, by applying  Theorem \ref{TH2.1}(ii),  we complete the proof of Corollary \ref{COR2.2}.
\end{proof}

\section{Proofs of Theorem \ref{TH3.1}}\hskip\parindent
In this section, we provide a proof for Theorem \ref{TH3.1}.
\begin{proof}[Proof of Theorem \ref{TH3.1}] We divide the proof into three parts.

(1) By the Markov property, the $\mu$-invariance of $(P_t^B)_{t\geq0}$, and the symmetry of $(P_t^B)_{t\geq0}$ in $L^2(\mu)$, we derive that,
for any $f\in L^2(\mu)\setminus\{0\}$,
\begin{equation*}\begin{split}
\E^\mu\big[f(X_{s_1}^B)f(X_{s_2}^B)\big]&=\E^\mu\big[f(X_{s_1}^B)P_{s_2-s_1}^B f(X_{s_1}^B)\big]\\
&=\mu\big[P_{s_1}^B(fP_{s_2-s_1}^B f)\big]=\mu(fP_{s_2-s_1}^B f)\\
&=\mu\big[(P_{\ff {s_2-s_1} 2}^B f)^2\big],\quad s_2\ge s_1\ge0.
\end{split}\end{equation*}
Hence
\begin{align*}
&\ff 1 t \E^\mu\left[\left|\int_0^t f(X_s^B)\,\d s\right|^2\right]=
\ff 2 t\int_0^t \d s_1\int_{s_1}^t\E^\mu\big[f(X_{s_1}^B) f(X_{s_2}^B)\big]\,\d s_2\\
&=\ff 2 {t}\int_0^t \d s_1\int_{s_1}^t \mu\big((P_{\ff{s_2-s_1} 2}^Bf)^2\big)\,\d s_2
=\ff 4 t \int_0^{t/2}\mu((P_s^B f)^2)\,\d s\int_s^{t-s}\d r\\
&= 4\int_0^{t/2}\Big(1-\frac{2s}{t}\Big)\mu\big((P_s^B f)^2\big)\,\d s,\quad t>0,
\end{align*}
where, in the third equality, we have used the  variables transformations, i.e., $s=\ff {s_2-s_1} 2,r=\ff {s_1+s_2} 2$; see \cite[Lemma 2.3]{CCG} and \cite[(3.2)]{Wang20202} for related results. Note that, as a function of $t$, the right hand side of the above identity is increasing. Letting $t\to\infty$, we have
\begin{equation}\label{TH3.1(1)}
\lim_{t\to\infty}\ff 1 t\E^\mu\left[\left|\int_0^t f(X_s^B)\,\d s\right|^2\right]=4\int_0^\infty\mu\big((P_s^B f)^2\big)\,\d s
\in(0,\infty].
\end{equation}
Taking $f\in L^2(\mu)\setminus\{0\}$ such that $\mu(f)=0$ and $\|f\|_\infty
\vee \|\nn f\|_\infty\le 1$, by the dual representation of the Wasserstein distance $\tilde{\W}_1$ (see e.g. \cite{Vi2003}), we get
\begin{equation}\label{dual-rep}
t\E^\mu[\tilde{\W}_1(\mu_t^B,\mu)^2]\ge \ff 1 t \E^\mu\left[\left|\int_0^t f(X_s^B)\,\d s\right|^2\right]
,\quad t>0.
\end{equation}
Combining \eqref{TH3.1(1)} and \eqref{dual-rep}, we immediately obtain the desired lower bound estimate \eqref{LB1} for some constant $c>0$.

Now assume  \eqref{SG}. Then, according to the mutual independence of $(X_t)_{t\geq0}$ and $(S_t^B)_{t\geq0}$, Minkowski's inequality and \eqref{BL},  we deduce that (see e.g. \cite{SSV2012})
\begin{equation}\begin{split}\label{Poin}
\|P_t^B f-\mu(f)\|_{L^2(\mu)}& \le \int_0^\infty \|P_s f-\mu(f)\|_{L^2(\mu)}\,\P_{S_t^B}(\d s)\\
&\le \int_0^\infty \e^{-\lambda_1 s}\|f-\mu(f)\|_{L^2(\mu)}\,\P_{S_t^B}(\d s)\\
&= \e^{-B(\ll_1) t} \|f-\mu(f)\|_{L^2(\mu)},\quad  t>0,\,f\in L^2(\mu).
\end{split}\end{equation}
Let $\nu\in \Pp$ such that $\nu=h_\nu\mu$ with $h_\nu\in L^2(\mu)$.  For every $f\in  L^2(\mu)\cap L^\infty(\mu)$  with $\mu(f)=0$, we have
\begin{equation}\label{lim0}\begin{split}
&\left|\ff 1 t \E^\nu\left[\left|\int_0^t f(X_s^B)\,\d s\right|^2\right]-
\frac{1}{t}\E^\mu\left[\left|\int_0^t f(X_s^B)\,\d s\right|^2\right]\right|\\
&=\ff 1 t\left|\int_M [h_\nu(x)-1] \E^x\left[\left|\int_0^ t f(X_s^B)\,\d s\right|^2\right]\,\mu(\d x)\right|\\
&=\ff 2 t\left|\int_0^t \d s_1 \int_{s_1}^t\mu\big([h_\nu-1]P_{s_1}^B[f P_{s_2-s_1}^B f]\big)\,\d s_2\right|\\
&=\ff 2 t \left|\int_0^t \d s_1 \int_{s_1}^t\mu\big([P_{s_1}^B(h_\nu-1)][f P_{s_2-s_1}^B f]\big)\,\d s_2\right|\\
&\leq\ff{2\|f\|_{L^\infty(\mu)}} t\int_0^ t \d s_1\int_{s_1}^t \|P_{s_1}^B(h_\nu-1)\|_{L^2(\mu)}
\|P_{s_2-s_1}^B f\|_{L^2(\mu)}\,\d s_2\\
&\leq \ff{2\|f\|_{L^\infty(\mu)}\|f\|_{L^2(\mu)}} t\big(\|h_\nu\|_{L^2(\mu)}+1\big)\int_0^t\d s_1\int_{s_1}^t \e^{-B(\lambda_1) s_2}\,\d s_2\\
&=2\|f\|_{L^\infty(\mu)}\|f\|_{L^2(\mu)}\big(\|h_\nu\|_{L^2(\mu)}+1\big)\frac{1-\e^{-B(\lambda_1)t}-B(\lambda_1)
\e^{-B(\lambda_1)t}t}{B(\lambda_1)^2t}\longrightarrow0,
\end{split}\end{equation}
as $t\rightarrow\infty$, where the third line follows from the Markov property, the fourth line is due to the symmetry of $(P_t^B)_{t\geq0}$ in $L^2(\mu)$, the fifth line is by the Cauchy--Schwarz inequality, and the sixth line is directly deduced from \eqref{Poin}. Taking $f\in L^2(\mu)\setminus\{0\}$ with $\mu(f)=0$ and $\|f\|_\infty\vee\|\nn f\|_\infty\le 1$, by \eqref{TH3.1(1)}, \eqref{dual-rep} and \eqref{lim0}, we obtain
\begin{equation}\begin{split}\label{TH3.1(2)}
\liminf_{t\to\infty} \left\{t\E^\nu[\tilde{\W}_1(\mu_t^B,\mu)^2]\right\}&\ge\liminf_{t\to\infty}
\left\{\ff 1 t \E^\nu\left[\left|\int_0^t f(X_s)\right|^2\right]\right\}\\
&=4\int_0^\infty\mu\big[(P_s^B f)^2\big]\,\d s>0.
\end{split}\end{equation}

Let $t\geq1$ and let $\bar{\mu}_t^B=\ff 1 t\int_1^{t+1}\dd_{X_s^B}\d s$. It is easy to see that
$$\pi:=\ff 1 t \int_0^1 \dd_{(X_s^B,X_{s+t}^B)}\,\d s+\ff 1 t \int_1^t \dd_{(X_s^B,X_s^B)}\,\d s\in\C(\mu_t^B,\bar{\mu}_t^B).$$
Hence
\begin{equation}\label{TH3.1(3)}
\tilde{\W}_1(\mu_t^B,\bar{\mu}_t^B)\le\int_{M\times M} \mathbbm{1}_{\{(u,v):u\neq v\}}(x,y)\,\pi(\d x,\d y)\leq\ff 1 t.
\end{equation}
Let $x\in M$ and  denote $\nu_x:=p_1^B(x,\cdot)\mu$.
Noting that $p_1^B(x,\cdot)\in L^2(\mu)$, by the
Markov property and \eqref{TH3.1(2)}, we have
$$\liminf_{t\to\infty}\big\{t\E^x[\tilde{\W}_1(\bar{\mu}_t^B,\mu)^2]\big\}=
\liminf_{t\to\infty}\big\{t\E^{\nu_x}[\tilde{\W}_1(\mu_t^B,\mu)^2]\big\}>0.$$
This together with \eqref{TH3.1(3)} and the triangle inequality for $\tilde{\W}_1$,  it is direct to have that
\begin{eqnarray*}
\liminf_{t\to\infty}\big\{t\E^x[\tilde{\W}_1(\mu_t^B,\mu)^2]\big\}
&\geq&\liminf_{t\to\infty}\Big\{\frac{t}{2} \E^x[\tilde{\W}_1(\bar{\mu}_t^B,\mu)^2]-t \E^x[\tilde{\W}_1(\bar{\mu}_t^B,\mu_t^B)^2]\Big\}\\
&\geq&\liminf_{t\to\infty}\Big\{\frac{t}{2} \E^x[\tilde{\W}_1(\bar{\mu}_t^B,\mu)^2]-\frac{1}{t}\Big\}>0.
\end{eqnarray*}
Thus, for every $\nu\in\Pp$, by Fatou's lemma, we have
\begin{equation*}\begin{split}
\liminf_{t\to\infty}\big\{t\E^\nu[\tilde{\W}_1(\mu_t^B,\mu)^2]\big\}&=\liminf_{t\to\infty}
\int_M t\E^x[\tilde{\W}_1(\mu_t^B,\mu)^2]  \,\nu(\d x)\\
&\ge \int_M \liminf_{t\to\infty}\big\{t\E^x[\tilde{\W}_1(\mu_t^B,\mu)^2]\big\}\,\nu(\d x)>0,
\end{split}\end{equation*}
which proves \eqref{LB2}.

(2) Let $t\ge 1$, $N\in \N$ and $0<p<\alpha\leq1$.  Let $B\in\mathbf{B}_\aa$.  Consider the empirical measure for the $B$-subordinated process $(X_t^B)_{t\geq0}$, i.e.,
$$\tilde{\mu}_N^B:=\ff 1 N \sum_{i=1}^N \dd_{X_{t_i}^B}=\ff 1 t\sum_{i=1}^N
\int_{t_i}^{t_{i+1}}\dd_{ X_{t_i}^B}\, \d s,$$
where $t_i:=\ff {(i-1)t} N,\,1\le i \le N$. It is clear that
$$\ff 1 t\sum_{i=1}^N \int_{t_i}^{t_{i+1}}\dd_{X_s^B}(\d x)\dd_{X_{t_i}^B}(\d y)\,\d s
\in \C(\mu_t^B,\tilde{\mu}_N^B).$$
Then
\begin{equation}\label{WpU}\tilde{\W}_p(\mu_t^B,\tilde{\mu}_N^B)\le\ff 1 t\sum_{i=1}^N\int_{t_i}^{t_{i+1}}
\big(\rho(X_s^B,X_{t_i}^B)\wedge 1\big)^p\,\d s.
\end{equation}

It is an elementary fact that, there exists a constant $\kappa>0$ such that
 $$B(t)\leq \kappa t^\alpha,\quad t\geq0.$$
By \eqref{BL}, we can find a constant $c_1>0$ such that
\begin{equation}\label{ment-S}\begin{split}
\E[(S_r^B)^p]&=\ff p {\Gamma(1-p)}\int_0^\infty(1-\e^{-r B(t)})t^{-p-1}\,\d t\cr
&\le\ff p {\Gamma(1-p)}\int_0^\infty (1-\e^{-\kappa rt^\aa})t^{-p-1}\,\d t\cr
&\le c_1 r^{\ff p \aa},\quad r\in[0,1].
\end{split}\end{equation}
By the proof of \cite[Theorem 1.3(2)]{Wang20202}, since $\mu(|\nabla V|)<\infty$, the following inequality
$$\E^\mu\big[ \big(\rho(X_0,X_t)\wedge 1\big)^2\big]\le c_2 t,\quad t\ge 0$$
holds for some constant $c_2>0$. Hence, by  H\"{o}lder's inequality, there exists a constant $c_3>0$ such that, for every $u\in[0,1]$,
\begin{equation}\label{DpU}\begin{split}
\E^\mu\big[ \big(\rho(X_0^B,X_u^B)\wedge 1\big)^p\big]&=\E^\mu\big[\big(\rho(X_0,X_{S_u^B})\wedge 1\big)^p\big]\cr
&\le c_2^{p/2}\E\big[(S_u^B)^{\ff p 2}\big]\le c_3 u^{\ff p {2\aa}},
\end{split}\end{equation}
where the last step follows from \eqref{ment-S} with $q=\ff p 2$.

Thus, by \eqref{WpU} and \eqref{DpU}, there exists a constant $c_4>0$ such that
\begin{equation}\label{TH3.1(4)}
\E^\mu[\tilde{\W}_p(\mu_t^B,\tilde{\mu}_N^B)]\le c_4(tN^{-1})^{\ff p {2\aa}}.
\end{equation}

Since \eqref{LRIC} holds, using the volume comparison theorem (see e.g. \cite[Proposition 3.5.9]{Wang2014} with $W=0$), we find a constant $c_5>0$ such that
$$\mu\big(\tilde{B}(x,r)\big)\le c_5 r^d,\quad x\in M,\,r\in(0,1],$$
where $\tilde{B}(x,r):=\{y\in M:\rho(x,y)\wedge 1\le r\}$. Note that this inequality holds for all $r>0$ since $\mu$ is a probability measure.
Then, according to \cite[Proposition 4.2]{K} (see also \cite[Corollary 12.14]{GL2000}),  we have
\begin{equation}\label{WpL}
\tilde{\W}_p(\tilde{\mu}_N^B,\mu)\ge c_6 N^{-\ff p d},
\end{equation}
for some constant $c_6>0$.

Thus, by the triangle inequality for $\tilde{\W}_p$, \eqref{TH3.1(4)} and \eqref{WpL} yield
\begin{align*}
\inf_{\nu\in\Pp_k}\E^\nu[\tilde{\W}_p(\mu_t^B,\mu)]&\ge\inf_{\nu\in\Pp_k}
\E^\nu[\tilde{\W}_p(\mu,\tilde{\mu}_N^B)]-\sup_{\nu\in\Pp_k}\E^\nu[\tilde{\W}_p
(\mu_t^B,\tilde{\mu}_N^B)]\\
&\ge c_6 N^{-\ff p d}-c_4 k (tN^{-1})^{\ff p {2\aa}}.
\end{align*}
Therefore, optimizing in $N\ge 1$, we obtain that there exist a constant $c_7>0$ such that
$$\inf_{\nu\in\Pp_k}\E^\nu[\tilde{\W}_p(\mu_t^B,\mu)]\ge c_7 k^{-\ff {2\aa}{d-2\aa}}t^{-\ff p {d-2\aa}},$$
which finishes the proof of \eqref{LB3} by the elementary fact that $\tilde{\W}_p\leq\tilde{\W}_1^p$ on $\Pp\times\Pp$ for every $p\in (0,1]$.

(3) According to \cite[Theorem 2.1(1)]{WangWu}, for any $\vv\in(0,1]$, we have
\begin{equation*}\label{TH3.1(5)}
\liminf_{t\to\infty}\big\{t\inf_{x\in M}\E^x[\W_2(\mu_{\vv,t}^B,\mu)^2]\big\}\ge
\sum_{i=1}^\infty \ff 2 {\ll_iB(\ll_i)\e^{2\vv\ll_i}}.
\end{equation*}
On the other hand, by \cite[Theorem 3.3.2]{Wang2014} (see also \cite{vRS2005} for the empty boundary case), since ${\rm Ric}-{\rm Hess}_V\ge K$ with $K\in\R$ and $\pp M$ is either empty or convex, it holds that
$$\W_2(\mu_{\vv,t}^B,\mu)^2=\W_2(\mu_t^B P_\varepsilon,\mu)^2\le \e^{-2\vv K}\W_2(\mu_t^B,\mu)^2,\quad \vv,t>0.$$
Thus, we have
$$\liminf_{t\to\infty}\big\{t\inf_{x\in M}\E^x[\W_2(\mu_t^B,\mu)^2]\big\}\ge \e^{2\vv K}
\sum_{i=1}^\infty\ff 2 {\ll_i B(\ll_i)\e^{2\vv \ll_i}},\quad \vv\in(0,1],$$
which immediately brings us \eqref{LB4} by letting $\vv\downarrow 0$.
\end{proof}

\section{Proofs of Example \ref{EXP4.1}}\hskip\parindent
In this section, we present the proof of Example \ref{EXP4.1}, which is achieved by adapting the proof of \cite[Example 1.4]{Wang20202} and applying our results in Corollary \ref{COR2.2} and Theorem \ref{TH3.1}.

\begin{proof}[Proofs of Example \ref{EXP4.1}]
Let $q>1$. It is easy to check that $\int_{\R^d}|x|^{2(q-1)}\e^{V(x)}\,\d x<\infty$, and hence $\mu(|\nabla V|^2)<\infty$.

(1) From the proof of \cite[Example 1.4]{Wang20202} on page 283, we see that \eqref{Ric} holds for some constant $K>0$. By  Corollary \ref{COR2.2}, it suffices to estimate  $\tilde{\eta}^\alpha$. By the proof of \cite[Example 1.4(1)]{Wang20202}, there exists a constant $c_1>0$ such that
$$\tilde{\gg}(u):=\int_{\R^d}\frac{1}{\mu\big(B(x,\sqrt{u})\big)}\,\mu(\d x)\le c_1 u^{-\ff {qd}{2(q-1)}},\quad u\in(0,1],$$
where $B(x,r):=\{y\in\R^d: |x-y|<r\}$, $x\in\R^d,\,r>0$. Then, we can find a constant $c_2>0$ such that,  for any $\vv\in(0,1]$,
\begin{equation*}\begin{split}
\tilde{\eta}^\alpha(\vv)&:=1+\int_\varepsilon^1\tilde{\gamma}(u)u^\alpha\,\d u\\
&\le 1+c_1\int_\vv^1 u^{-\ff{qd}{2(q-1)}}u^\aa\, \d u\\
&\le
\begin{cases}
c_2 \vv^{1+\aa-\ff{qd}{2(q-1)}},\quad &\mbox{if}~1+\aa<\ff{qd}{2(q-1)},\\
c_2 \log(1+\vv^{-1}),\quad &\mbox{if}~1+\aa=\ff{qd}{2(q-1)},\\
c_2,\quad &\mbox{if}~1+\aa>\ff{qd}{2(q-1)}.
\end{cases}
\end{split}\end{equation*}
Optimizing in $t>0$  separately, we obtain that
\begin{equation}\begin{split}\label{EXP4.1(1)}
\inf_{\vv\in(0,1]}\big\{\vv+t^{-1}\tilde{\eta}^\alpha(\vv)\big\}\le
\begin{cases}
ct^{-\ff{2(q-1)}{(d-2\aa)q+2\aa}},\quad &\mbox{if}~1+\aa<\ff{qd}{2(q-1)},\\
ct^{-1}\log(1+t),\quad &\mbox{if}~1+\aa=\ff{qd}{2(q-1)},\\
ct^{-1},\quad &\mbox{if}~1+\aa>\ff{qd}{2(q-1)},\\
\end{cases}
\end{split}\end{equation}
for some constant $c>0$. Therefore, Corollary \ref{COR2.2} implies \eqref{EXPUB1}.

(2) By the proof of \cite[Example 1.4(3)]{Wang20202} on page 284, the spectra gap inequality \eqref{SG} holds for some constant $\lambda_1>0$ and \eqref{LRIC} is true for some constant $K\geq0$.  Thus, the desired assertion follows from Theorem \ref{TH3.1}(2).
\end{proof}

Finally, we give a remark on the ultra-contractivity of the subordinated semigroup considered in the above example.
\begin{remark}
Let $q>2$ and $B\in \mathbf{B}^\aa$ for some $\aa\in(\ff q {2q-2},1]$. We claim that $(P_t^B)_{t>0}$ is ultra-contractive. Indeed, it is easy to see that,  there exists a constant $c>0$ such that
$$\int_1^\infty \ff {\d r}{B\big(r^{2-\ff 2 q}\big)}\le c\int_1^\infty \ff {\d r}{r^{\aa(2-\ff 2 q)}}<\infty.$$
Then, from the proof of \cite[Example 1.4(2)]{Wang20202}, we can find a constant $C>0$ such that
$$\|P_t\|_{L^1(\mu)\to L^\infty(\mu)}\le \exp\Big[C\big(1+t^{-\frac{q}{q-2}}\big)\Big],\quad t>0. $$
Thus, by \cite[Proposition 13]{SchWang}, the claim is proved.
\end{remark}

\subsection*{Acknowledgment}\hskip\parindent
The authors would like to acknowledge the referee for corrections and helpful comments, and thank Prof. Feng-Yu Wang for helpful conversations and Dr. Jie-Xiang Zhu for useful comments and corrections on the former edition of the paper. This work is supported by the National Natural Science Foundation of China (Grant No. 11831014).

\section*{Appendix}\hskip\parindent
 In the appendix, we prove Remark \ref{remark-1}. The proof may be familiar for experts. However, we present it here for completeness.
 \begin{proof} Firstly, for every $x\in M$, $t\mapsto p_t(x,x)$ is decreasing in $(0,\infty)$. Indeed, by the symmetry, the semigroup property and the contraction property, for every $0<s<t<\infty$,
\begin{equation}\label{A1}
p_t(x,x)=\|p_{\frac{t}{2}}(x,\cdot)\|_{L^2(\mu)}^2=\|P_{\frac{t-s}{2}}p_{\frac{s}{2}}(\cdot,x)\|_{L^2(\mu)}^2
\leq\|p_{\frac{s}{2}}(\cdot,x)\|_{L^2(\mu)}^2=p_s(x,x).
\tag{A1}\end{equation}

Secondly, for every $t\geq t_0/2$,
$$\lim_{N\rightarrow\infty}\sup_{\|f\|_{L^2(\mu)}\leq1}\|P_tf\mathbbm{1}_{\{|P_tf|\geq N\}}\|_{L^2(\mu)}=0.$$
Indeed, letting $A_{N}=\{|P_tf|> N\}$ for each $N\in\mathbb{N}$, every $f\in L^2(\mu)$ and every $t>0$, by Minkowski's inequality, the Cauchy--Schwarz inequality, Fubini's theorem and properties of $(p_t)_{t>0}$,  we have, for every $f\in L^2(\mu)$ and every $t>0$,
\begin{align*}
\|P_tf\mathbbm{1}_{\{|P_tf|\geq N\}}\|_{L^2(\mu)}&=\Big\{\int_M\Big(\int_Mf(y)p_t(x,y)\,\mu(\d y)\Big)^2\mathbbm{1}_{A_N}(x)\,\mu(\d x)\Big\}^{1/2}\\
&\leq\int_M\Big(\int_M f(y)^2p_t(x,y)^2\mathbbm{1}_{A_N}(x)\,\mu(\d x)\Big)^{1/2}\,\mu(\d y)\\
&\leq\Big\{\int_M\Big(\int_M p_t(x,y)^2\mathbbm{1}_{A_N}(x)\,\mu(\d x)\Big)\,\mu(\d y)\Big\}^{1/2}\|f\|_{L^2(\mu)}\\
&=\Big\{\int_M\Big(\int_M p_t(x,y)^2\,\mu(\d y)\Big)\mathbbm{1}_{A_N}(x)\,\mu(\d x)\Big\}^{1/2}\|f\|_{L^2(\mu)}\\
&=\Big(\int_Mp_{2t}(x,x)\mathbbm{1}_{A_N}(x)\,\mu(\d x)\Big)^{1/2}\|f\|_{L^2(\mu)}.
\end{align*}
Let $A_N^\ast=\{\sup_{\|f\|_{L^2(\mu)}\leq1}|P_t f|>N\}$ for every $N\in\mathbb{N}$ and every $t>0$. By the Cauchy--Schwarz inequality,
\begin{equation*}\begin{split}
&\int_M \sup_{\|f\|_{L^2(\mu)}\leq1}|P_t f|\,\d\mu=\int_M \sup_{\|f\|_{L^2(\mu)}\leq1}\Big|\int_M f(y)p_t(x,y)\,\mu(\d y)\Big|\,\mu(\d x)\\
&\leq\int_M \Big(\int_M p_t(x,y)^2\,\mu(\d y)\Big)^{1/2}\,\mu(\d x)\leq\sqrt{\gamma(2t)}<\infty,\quad t\geq t_0/2,
\end{split}\end{equation*}
which implies that $\sup_{\|f\|_{L^2(\mu)}\leq1}|P_t f|<\infty$~$\mu$-a.e., $t\geq t_0/2$. Then for every $t\geq t_0/2$, $\mathbbm{1}_{A_N^\ast}\rightarrow0$~$\mu$-a.e. as $N\rightarrow\infty$.  By Fatou's lemma, we obtain that
\begin{equation*}\begin{split}
&\lim_{N\rightarrow\infty}\sup_{\|f\|_{L^2(\mu)}\leq1}\|P_tf\mathbbm{1}_{\{|P_tf|\geq N\}}\|_{L^2(\mu)}\\
&\leq \limsup_{N\rightarrow\infty}\sup_{\|f\|_{L^2(\mu)}\leq1}\Big(\int_Mp_{2t}(x,x)\mathbbm{1}_{A_N}(x)\,\mu(\d x)\Big)^{1/2}\\
&\leq\limsup_{N\rightarrow\infty}\Big(\int_Mp_{2t}(x,x)\mathbbm{1}_{A_N^\ast}(x)\,\mu(\d x)\Big)^{1/2}\leq0,\quad t\geq t_0/2,
\end{split}\end{equation*}
where the first inequality in the last line is due to that $\sup_{\|f\|_{L^2(\mu)\leq1}}\mathbbm{1}_{A_N}\leq \mathbbm{1}_{A_N^\ast}$ for any $t>0$ since $A_N\subset A_N^\ast$ for every $f\in L^2(\mu)$ with $\|f\|_{L^2(\mu)}\leq1$ and every $t>0$.

Finally, due to \cite[Corollary 1.6.9 and Corollary 1.6.6]{Wang2014}, we deduce that the essential spectrum of $L$ is empty, which finishes the proof.
\end{proof}

\end{document}